\documentclass[12pt]{article}
\usepackage[english]{babel}
\usepackage[cp1251]{inputenc}
\usepackage{amsmath,amssymb,graphicx,amsthm,enumerate,comment,amscd}
\usepackage[unicode,pdftex]{hyperref}
\usepackage[matrix,arrow,curve]{xy}

\newtheorem{theorem}{Theorem}

\newtheorem{corollary}{Corollary}

\newtheorem{proposition}{Proposition}
\newtheorem{question}{Question}

\theoremstyle{remark}
\newtheorem{remark}{Remark}

\theoremstyle{definition}

\newcommand{\R}{\mathbb{R}}

\newcommand{\Z}{\mathbb{Z}}
\renewcommand{\SS}{\mathbb{S}}
\newcommand{\del}{\smash{\mskip3mu\lower1truept\hbox{\vdots}\mskip3mu}}

\renewcommand{\phi}{\varphi}

\renewcommand{\kappa}{\varkappa}

\begin{document}

\author{Andrey Ryabichev%
\footnote{\  Yu.I.Manin Laboratory of Algebra, Geometry, Logic and Number Theory, Higher School of Modern Mathematics, MIPT; Moscow, Russia //
\texttt{ryabichev@179.ru}
%Supported in part by Young Russian Mathematics award
}}
\title{On vertex-minimal simplicial maps to the sphere}
\date{}

%\email{ryabichev@179.ru}

\maketitle

\begin{abstract}
For positive integers $n,d$, let $\lambda(n,d)$ be the minimal number of vertices of a triangulation of the $n$-sphere
which admits a degree $d$ simplicial map onto the boundary of the $(n+1)$-simplex.
We show that for $h=\lfloor\frac{n+1}2\rfloor$, the function $\lambda(n,d)^h$ 
%is asymptotically linear in $d$,
%as $d\to\infty$,
has linear order of growth in $d$,
answering a question of O.\,Musin.
%the conjecture of  K.\,Apolonskaya and O.\,Musin from \href{https://arxiv.org/abs/2511.10870}{arXiv:2511.10870}.
All triangulations we obtained are isomorphic to boundaries of convex polytopes in $\R^{n+1}$.
\end{abstract}

\section{Introduction}\label{s:introduction}

Suppose $f:M\to N$ is a continuous map of closed connected oriented $n$-manifolds, $n\ge1$.
The {\it degree} $\deg f$ is an integer $d$ such that the induced homomorphism on $n$-th homology groups
$f_*:H_n(M;\Z)=\Z\to\Z=H_n(N;\Z)$ is multiplication by $d$.
If the map $f$ is smooth, then $\deg f$ equals the number of preimages of any regular value of $f$,
where the points at which $df$ preserves the orientation are counted as ``$1$'',
and the points at which $df$ reverses the orientation are counted as ``$-1$''.
If $f$ is a simplicial map of triangulated manifolds,
then the degree can be counted in a similar way via preimages of an $n$-simplex.

Namely, we call an $n$-simplex $\sigma\subset M$ {\it positive/negative},
if $f(\sigma)\subset N$ is a non-degenerate $n$-simplex and $f|_\sigma$ preserves/reverses orientation.
Then $\deg f$ equals the number of positive preimages minus the number of negative preimages of any $n$-simplex $\tau\subset N$.
This difference does not depend on $\tau$.

Let $\SS_{n+2}^n$ be the standard triangulation of the $n$-sphere with $n+2$ vertices
and suppose $K$ is a simplicial complex such that the geometric realization $|K|$ is homeomorphic to the sphere~$\SS^n$.
Suppose there is a simplicial map $K\to\SS^n_{n+2}$ of degree $d$.
One may ask: {\it given $n$ and $d$, how few vertices can $K$ have?}
Denote the minimal number of vertices of such $K$ by $\lambda(n,d)$.

Evidently, if $d=0$ or $d=\pm1$, we have $\lambda(n,d)=n+2$,
since a triangulation of the $n$-sphere cannot have fewer vertices.
It is also easy to see that $\lambda(1,d)=3d$.
For $n=2$ and $d\ge3$ it is known from \cite{sarkaria} that $\lambda(2,d)=2d+2$
(the lower bound for $\lambda(2,d)$ obviously follows from the Euler's formula:
$K$ has at least $4d$ faces, so the number of vertices is at least $2+6d-4d$).

In \cite{basak-gupta-triverdi},
%B.\,Basak, R.\,K.\,Gupta and A.\,Trivedi
Basak, Gupta and Trivedi
take the first steps in studying $\lambda(n,d)$.
In \cite{musin-deg},
%K.\,Apolonskaya and O.\,Musin
Apolonskaya and Musin
found values of $\lambda(n,d)$
%for $d\le n$.
for small $d$.
They also ask the question,
{\it for what smallest $\omega$ does $\lim\sup_{d\to\infty}\frac{\lambda(n,d)}{d^\omega}$ exist?}
We prove that $\omega=1/h$, where $h=\lfloor\frac{n+1}2\rfloor$.
This generalizes the facts mentioned above.

\begin{theorem}\label{th:lambda-in-h}
For any $n\ge1$, the value $\lambda(n,d)^h$ is asymptotically linear in $d$ (as $d\to\infty$).
More precisely, there are $c_1,c_2>0$ such that $c_1d\le\lambda(n,d)^h\le c_2d$ for $d\gg0$.
\end{theorem}

\begin{corollary}\label{cor:limsup-lambda}
$\lim\sup_{d\to\infty}\frac{\lambda(n,d)}{d^\omega}$
is zero for $\omega>1/h$,
exists and is positive for $\omega=1/h$,
and 
%does not exist 
tends to $+\infty$
for $\omega<1/h$.
\end{corollary}

%\bigskip

Note that the bounds from Theorem~\ref{th:lambda-in-h}
hold even if we
consider only {\it combinatorial} triangulations
%triangulations.
or, moreover, only {\it polytopal} triangulations.
%Moreover, we can consider only triangulations which are {\it isomorphic to the boundary of a convex simplicial polytope in $\R^{n+1}$}.
See Remark~\ref{r:comb-nondegenerate}.

\section{Preliminaries}

A {\it simplicial complex} $K$ is a set of vertices $V(K)$ and a collection of its finite subsets (called {\it faces})
such that any subset of a face is itself a face.
We always denote the face with vertices $a_1,\ldots,a_k$ by $[a_1\ldots a_k]$.
A {\it simplicial map} of complexes is a map of the sets of vertices such that the image of a face is a face.

One can think of a face as a (topological) simplex with given vertices,
the union of such simplices (glued along common faces and equipped with the quotient topology),
is called {\it the geometric realization} $|K|$.
Everywhere below we make no difference between an abstract complex and its geometric realization.
If $|K|$ is an $n$-manifold, denote by $F(K)$ the set of its $n$-simplices.

For simplicial complexes $K,L$,
their {\it join} is a complex $K\star L$
with vertices $V(K\star L)=V(K)\sqcup V(L)$
whose simplices are all simplices of $K$, all simplices of $L$,
and all simplices of the form $[a_1\ldots a_k]\star[b_1\ldots b_l]:=[a_1\ldots a_k\, b_1\ldots b_l]$,
where $[a_1\ldots a_k]\in K$ and $[b_1\ldots b_l]\in L$.

\begin{remark}\label{r:join-faces}
If $|K|\cong\SS^n$ and $|L|\cong\SS^m$, where $n,m\ge0$, then $|K\star L|\cong\SS^{n+m+1}$.
Note that $F(K\star L)=F(K)\times F(L)$ by definition.
\end{remark}

For details on simplicial complexes, see, e.\,g., \cite[\S1 and \S4]{matousek-discr},
and for a discussion of the degree, see, e.\,g., \cite[\S2.2]{hatcher}.

\begin{proposition}\label{pr:star-correct}
Suppose $|K|\cong|K'|\cong\SS^k$ and $|L|\cong|L'|\cong\SS^m$.
Take simplicial maps $f:K\to K'$ and $g:L\to L'$.
Then there is a naturally defined map $f\star g:K\star L\to K'\star L'$,
and $\deg(f\star g)=\deg f\cdot\deg g$.
\end{proposition}

\begin{proof}
The map $f\star g$ is defined on the set of vertices, the image of a face is a face by definition.
Note that the orientation of a simplex $[a_1\ldots a_{k+1}]$ is determined by the order of its vertices
and does not change under even permutation of these vertices.
Therefore the orientations of $K$ and $L$ determine the orientation of $K\star L$.
Finally, the sign of the simplex $[a_1\ldots a_{k+1}]\star[b_1\ldots b_{m+1}]$ under the map $f\star g$
%Finally, if $f\star g([a_1\ldots a_{k+1}]\star[b_1\ldots b_{m+1}])$ is a simplex of dimension $k+m$ in $K'\star L'$, then its image
equals the product of the signs of $[a_1\ldots a_{k+1}]\in K$ under $f$ and $[b_1\ldots b_{m+1}]\in L$ under $g$,
so the proposition follows.
\end{proof}

\begin{proposition}[Edge contraction]\label{pr:deg1}
For any $k,m\ge0$, there is a degree $1$ simplicial map $\SS^k_{k+2}\star\SS^m_{m+2}\to\SS^{k+m+1}_{k+m+3}$.
\end{proposition}

\begin{proof}
Denote the set
$V(\SS^k_{k+2})$ as $\{a_1,\ldots,a_{k+2}\}$, the set
$V(\SS^m_{m+2})$ as $\{b_1,\ldots,b_{m+2}\}$ and the set
$V(\SS^{k+m+1}_{n+m+3})$ as $\{c_1,\ldots,c_{k+m+3}\}$.
Define the map $f:V(\SS^k_{k+2}\star\SS^m_{m+2})\to V(\SS^{k+m+1}_{k+m+3})$ as
$a_1\mapsto c_1$, \ldots,
$a_{k+2}\mapsto c_{k+2}$ and
$b_1\mapsto c_{k+2}$, \ldots,
$b_{m+2}\mapsto c_{k+m+3}$.
Clearly, $f^{-1}([c_1\ldots c_{k+1} c_{k+3}\ldots c_{k+m+3}])$
is just one simplex $[a_1\ldots a_{k+1}]\star[b_2\ldots b_{m+2}]$,
therefore $\deg f=\pm1$.
If needed, swap two vertices to make it equal~1.
\end{proof}

%\newpage

\section{The examples}

In order to find $c_2$ in Theorem~\ref{th:lambda-in-h}, we will construct an example of a simplicial map with few vertices.
This construction uses two ideas.
The first idea is to take a join of simplicial maps (Proposition~\ref{pr:star-correct}),
%the numbers of vertices will be added together, while the degrees will be multiplied.
the numbers of vertices add, whereas the degrees multiply.
The second idea is to increase the degree a little by adding a limited number of vertices, in the following way.

\begin{proposition}\label{pr:d+1}
For any $n>0$ and $d>0$ we have $\lambda(n,d+1)\le\lambda(n,d)+3$.
\end{proposition}

\begin{proof}
Take a simplicial map $f:K\to\SS^n_{n+2}$ of degree $d$.
For a positive $n$-simplex $\sigma\subset K$
we make a {\it stellar subdivision}:
we add one more vertex to $V(K)$ and replace $\sigma$ by the cone over $\partial\sigma$,
which consists of $n$-simplices $\sigma_1,\ldots,\sigma_{n+1}$
(so $|K|$ does not change up to homeomorphism).
Note that $f$ can be uniquely extended to $\sigma_1,\ldots,\sigma_{n+1}$
so that they become non-degenerate negative simplices.
Therefore the extended map has degree $d-1$.

\begin{figure}[h]
\center{\includegraphics[width=130mm]{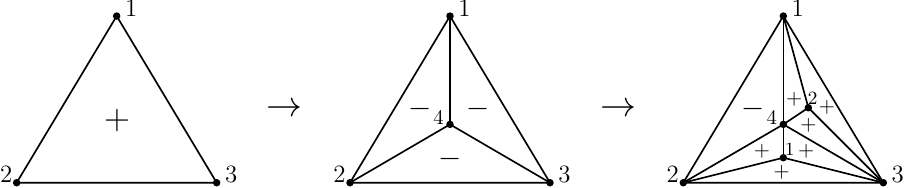}}
\caption{%
%The subdivisions for $n=2$.
Labels on vertices indicate their images,
and labels on simplices indicate their signs}\label{fig:central-subdivision}
\end{figure}

Now make stellar subdivisions of $\sigma_1$ and $\sigma_2$
and extend the map $f$ to all the new simplices in a similar way
(Fig.~\ref{fig:central-subdivision} illustrates the case $n=2$).
Denote the obtained complex by $A(K)$.
We have $|V(A(K))|=|V(K)|+3$,
and the degree of the obtained map $A(K)\to\SS^n_{n+2}$ equals $d+1$.
%The degree of the map $f$ extended to the all new simplices in the similar way equals $d+1$.
\end{proof}

Let $\SS^1_m$ denote the triangulation of $\SS^1$ with $m$ vertices.
If we apply the operation from Proposition~\ref{pr:d+1}
to a standard degree $k$ map $\SS^1_{3k}\to\SS^1_3$,
then $A(\SS^1_{3k})=\SS^1_{3k+3}$
and we obtain the standard degree $k+1$ map $\SS^1_{3k+3}\to\SS^1_3$.

\begin{theorem}\label{th:6hh-example}
For any $d>\big(\frac1{\sqrt[h]2-1}+1\big)^h$
there is a triangulated sphere $K$ with $|V(K)|\le 6h\sqrt[h]d$ and a degree $d$ map $K\to\SS^n_{n+2}$.
\end{theorem}

\begin{proof}
Let $l:=\lfloor\sqrt[h]d\rfloor$.
In other words, $l^h\le d<(l+1)^h$.
The condition $d>\big(\frac1{\sqrt[h]2-1}+1\big)^h$
is needed to deduce that
\begin{equation}\label{eq:2lh}
2l^h>(l+1)^h.
\end{equation}
Indeed, $(l+1)^h>d>\big(\frac1{\sqrt[h]2-1}+1\big)^h$,
therefore $l>\frac1{\sqrt[h]2-1}$,
so $\frac1l<\sqrt[h]2-1$,
then $\frac{l+1}l<\sqrt[h]2$,
and we obtain (\ref{eq:2lh}).
In particular, we have $l\ge2$.

Write $d=:\overline{1k_1\ldots k_h}$ in the base-$l$ number system
(it has $h+1$ digits, and from (\ref{eq:2lh}), the most significant digit is ``$1$'').
This precisely means that
\begin{equation}\label{eq:skobki}
d=(\ldots((l+k_{1})l+k_{2})l+\ldots+k_{h-1})l+k_h,
\end{equation}
where $0\le k_1,\ldots,k_{h}<l$.

\medskip

Suppose $n$ is odd.
Next we inductively define a sequence of simplicial complexes $X_1$, $Y_1$, $X_2$, $Y_2$,\,\ldots, $X_h$, $Y_h$.
Let $X_1:=\SS_{3l}$,
then $Y_i:=A^{k_i}(X_i)$ for $i=1,\ldots,h$,
and $X_j:=Y_{j-1}\star\SS^1_{3l}$ for $j=2,\ldots,h$.
Note that $|X_i|\cong|Y_i|\cong\SS^{2i-1}$.

We have a degree $l$ map $X_1\to\SS^1_3$.
Then using Proposition~\ref{pr:d+1} we obtain a degree $l+k_1$ map $Y_1\to\SS^1_3$.
Further, by Propositions \ref{pr:star-correct} and \ref{pr:deg1},
we construct a degree $(l+k_1)l$ map $X_2\to\SS^3_5$.
Then again by Proposition~\ref{pr:d+1} we get a degree $(l+k_1)l+k_2$ map $Y_2\to\SS^3_5$.
Continuing in this way according to the sequence of summands and factors in expression (\ref{eq:skobki}),
we will eventually obtain a degree $d$ map $Y_h\to\SS^n_{n+2}$.

To complete the proof in the case of odd $n$
it remains to note that $|V(Y_i)|=|V(X_i)|+3k_i$ and $|V(X_{i+1})|=|V(Y_i)|+3l$,
so $|V(Y_h)| = 3lh+3\sum k_i<6lh$.
Therefore $|V(Y_h)|<6h\sqrt[h]d$ and we can set $K=Y_h$.

\medskip

To construct an example in the case of even $n$, we just take the suspension, $K=Y_h\star\SS^0$.
Note that $k_1<l$, so we have $3\sum k_i+2\le3lh$, therefore the estimation $|V(K)|\le6lh$
for the complex with two added vertices still holds.
\end{proof}

\begin{remark}\label{r:comb-nondegenerate}
Note that
%the triangulations constructed in lemma~\ref{l:3-est} and corollary~\ref{c:n-est} are combinatorial,
this triangulation
%, as well as a triangulation from theorem~\ref{th:lambda-to-0},
is combinatorial
(the star of every vertex admits
%simplicial
a $PL$-%
embedding into $\R^n$, so the stars form a $PL$-atlas of~$\SS^n$),
since the join of combinatorial triangulations of spheres is a combinatorial triangulation of a sphere,
and the stellar subdivision of a combinatorial triangulation remains combinatorial.

Moreover, these triangulations are combinatorially equivalent to the boundary of a convex simplicial polytope in $\R^{n+1}$.
Indeed, suppose $P\subset\R^{p+1}$ and $Q\subset\R^{q+1}$
are convex polytopes containing the origins in their interiors.
Embed them into $\R^{p+q+2}$
as $P\times 0$ and $0\times Q$.
%$\R^{k+1}\times0\subset\R^{k+m+2}$ and $0\times\R^{m+1}\subset\R^{k+m+2}$.
Then the join $\partial P\star\partial Q$
%is canonically embedded into $\R^{k+m+2}=\R^{k+1}\times\R^{m+1}$
is just the boundary of the convex hull of the union $P\cup Q$
(called the {\it free sum} of the polytopes).
And the stellar subdivision can be obtained by adding a new vertex outside the polytope near its facet and taking the convex hull.
%It remains to notice that $S$ is the boundary of the cone over $S$ with the vertex at the origin
%(which in fact coincides with $\mathrm{conv}(S)\subset\R^{k+m+2}$).
\end{remark}

\section{The lower bound}\label{s:cyclic-polytope}

To find $c_1$ in Theorem~\ref{th:lambda-in-h}, we need the following facts
(see, e.\,g., \cite[\S5]{matousek-discr} or \cite{ziegler} for more details).

Let $\gamma:\R\to\R^{n+1}$ be the moment curve, $\gamma(t):=(t,t^2,\ldots,t^{n+1})$.
Take any $0<t_1<\ldots<t_k$ and denote the convex hull of $\gamma(t_1),\ldots,\gamma(t_k)\in\R^{n+1}$ by $P_k$.
It is called a {\it cyclic polytope}; it is a simplicial $(n+1)$-polytope with $k$ vertices as soon as $k\ge n+2$.
The set of its facets defines a triangulation of $\SS^{n}$, which does not depend on the choice of $t_1,\ldots,t_k$.

\begin{theorem}[Upper bound, \cite{stanley}]\label{th:upper-bound}
If $|K|\cong\SS^n$ and $|V(K)|=k$, then $|F(K)|\le |F(\partial P_k)|$.
%\hfill$\square$
\end{theorem}

\begin{theorem}[Gale evenness condition, \cite{gale}]\label{th:evenness}
The vertices $\gamma(t_{i_1}),\ldots,\gamma(t_{i_{n+1}})\in P_k$ form a face if and only if
for every $j,j'\notin\{i_1,\ldots,i_{n+1}\}$ the number of indices $\{i_s\ |\ j<i_s<j'\}$ is even.
%More precisely,
Namely,
$|F(\partial P_k)|=2{k-h-1\choose h}$ for even $n$ and
$|F(\partial P_k)|={k-h\choose h}+{k-h-1\choose h-1}$ for odd $n$.
%\hfill$\square$
\end{theorem}

\begin{corollary}\label{cor:2k-in-h}
For all $n>0$ and $k\ge n+2$
%and $k\ge n+2$
we have $|F(\partial P_k)|\le\frac{2k^h}{h!}$.
\end{corollary}

We use these results in the following estimate.

\begin{proposition}\label{pr:lowbound}
If $|K|\cong\SS^n$ and there is a degree $d$ simplicial map $f:K\to\SS^n_{n+2}$,
then $|V(K)|^h\ge d\cdot\frac{(n+2)h!}2$.
\end{proposition}

\begin{proof}
Suppose $K$ has $k$ vertices.
By Theorem~\ref{th:upper-bound} we have $|F(K)|\le |F(\partial P_k)|$.
%By Theorem~\ref{th:evenness} we have
%$|F(P_k)|\le2{k-h\choose h}$ for all.
%Therefore
Then by Corollary~\ref{cor:2k-in-h}
we have $|F(K)|\le\frac{2k^h}{h!}$.

On the other hand,
the preimage of every $n$-face of $\SS^n_{n+2}$ under $f$
consists of at least $d$ non-degenerate $n$-simplices.
So $|F(K)|\ge d(n+2)$, and the proposition follows.
\end{proof}

\begin{proof}[Proof of Theorem~\ref{th:lambda-in-h}]
Let $c_1=\frac{(n+2)h!}2$. Then Proposition~\ref{pr:lowbound} gives $\lambda(n,d)^h\ge c_1d$ for all $d$.
Further, let $c_2=(6h)^h$. Then Theorem~\ref{th:6hh-example} gives $\lambda(n,d)^h\le c_2d$ for all $d>\big(\frac1{\sqrt[h]2-1}+1\big)^h$.
\end{proof}

\section{Questions}

%We have estimated
We have obtained linear
upper and lower bounds for $\lambda(n,d)^h$ as $d\to\infty$.
However, finding exact values of $\lambda(n,d)$ appears to be a difficult problem.
Our estimates
%for $\lambda(n,d)$ obtained in the proof of theorem~\ref{th:lambda-to-0}
are apparently far from  optimal.
It is also interesting to study the monotonicity of this function.

\begin{question}
Does $\lim_{d\to\infty}\frac{\lambda(n,d)^h}d$ exist?
\end{question}

Note that the examples of maps we constructed have a lot of degenerate simplices.
The problem of obtaining a similar estimation
for maps without degenerate simplices is more difficult.

\begin{question}
Would the statement of Corollary~\ref{cor:limsup-lambda}
%the estimates from theorem~\ref{th:lambda-in-h}
remain correct if we consider only maps with no degenerate $n$-simplices?
What if we consider maps that have only positive $n$-simplices?
\end{question}

It is proved in \cite[Cor.\,27]{basak-gupta-triverdi} that for $d\gg0$ there is a degree $d$ simplicial map $\SS^n\to\SS^n_{n+2}$ that has only positive $n$-simplices.
However, in their example, the number of vertices of the triangulation of the domain sphere also depends linearly on $d$. 
A list of other interesting problems is also provided in \cite{basak-gupta-triverdi}.

%A list of other interesting problems is given in \cite{basak-gupta-triverdi}.

\appendix

\section{Appendix. Simplicial maps of the cyclic polytope}\label{s:appendix-cyclyc}

%One can try to construct a simplicial map from the boundary of a cyclic polytope.
%It would be interesting to study whether it is degree-extreme.

Since the maximum number of faces of a triangulation of the sphere
(for a fixed number of vertices) is achieved by a cyclic polytope,
it is natural to try to construct a map of maximal degree using $\partial P_k$.

Let $K=\partial P_k$ with vertices $v_1,\ldots,v_k$,
%$k\gg n+1$.
and let $V(\SS^n_{n+2})=\{a_1,\ldots,a_{n+2}\}$.
Consider the map $\phi_k:K\to\SS^n_{n+2}$ defined as $v_i\mapsto a_j$ for $j\equiv i\pmod{n+2}$.

\begin{question}
For which $n$ and $k$ is the degree of $\phi_k$ maximal among all maps from simplicial $n$-spheres with $k$ vertices to $\SS^n_{n+2}$?
\end{question}

Qihang Wang \cite{wang} derived the following formula for $\deg\phi_k$.
Suppose $k=q(n+2)+s$, where $q\ge1$ and $0\le s\le n+1$.
Let $j=\max\{0,\lfloor\frac{s-1}2\rfloor\}$ if $n$ is odd
and let $j=\max\{0,\lfloor\frac{s-2}2\rfloor\}$ if $n$ is even.

\begin{theorem}[{\cite[Th.\,1.1]{wang}}]
We have $\deg \phi_k=q^{h-j}(q+1)^j$,
except in the case where $n$ is even and $k$ is odd, when $\deg \phi_k=0$.
\end{theorem}

In particular, if $n$ is odd or $k$ is even, then $\deg \phi_k\sim \frac1{(n+2)^h}k^h$ as $k\to\infty$.
This asymptotic estimate of the degree 
is better than the estimate $\deg \ge\frac1{(6h)^h}k^h$ for the examples from Theorem~\ref{th:6hh-example},
but worse than the estimate $\deg f\le\frac2{(n+2)h!}k^h$ in Corollary~\ref{cor:2k-in-h}.

However, Wang showed \cite[Th.\,1.2]{wang} that even for $k=n+6$ the map $\phi_k$ does not have maximal degree for all $n$.
The problem whether $\phi_k$ can have maximal degree for arbitrarily large $k$ remains open.

\subsection*{Acknowledgments}

I wish to thank Mikhail Bludov, Roman Karasev, and Fedor Vylegzhanin for consultations on literature,
as well as to O.\,Musin for his attention to the work.
I am also very grateful to Mikhail Bludov, Roman Karasev and Anastasia Vakhrina for verifying the proof,
and to the anonymous referee for helpful advice.
Finally, I am grateful to participants of the MIPT seminar ``Combinatorics and topology'' for productive discussions.

The research is supported by MSHE RF GZ project.
%The research is supported by the MSHE.
%project No. FSMG-2024-0048.

%\newpage

\end{document}